\documentclass[a4paper,12pt]{article}%
\usepackage{amsmath}%
\usepackage{amsfonts}%
\usepackage{amssymb}%
\usepackage{graphicx}

\newtheorem{lemma}{Lemma}
\newtheorem{proposizione}{Proposition}

\newtheorem{definizione}{Definition}
\newenvironment{proof}[1][Proof]{\textbf{#1.} }{\ \rule{0.5em}{0.5em}}

\begin{document}
\title{DD-DA  PinT-based model:\\ A Domain Decomposition approach in space and time, based on Parareal, for solving  the 4D-Var Data Assimilation model }

\author{Luisa D'Amore, Rosalba Cacciapuoti \\Department of Mathematics and Applications \\  University of Naples Federico II, Naples, ITALY \\ luisa.damore@unina.it, rosalb.cacciapuoti@studenti.unina.it}
\date{}
\maketitle

\begin{abstract}
\noindent We present the mathematical framework of a Domain Decomposition (DD) aproach based on Parallel-in-Time methods (PinT-based approach) for solving the 4D-Var Data Assimilation (DA) model. The main outcome  of the proposed DD PinT-based approach is: 
\begin{enumerate}
\item  DA acts as coarse/predictor for the local PDE-based forecasting model, increasing the accuracy of the local solution. 
\item The fine and coarse solvers can be used in parallel, increasing the efficiency of the algorithm.  \item Data locality is preserved and data movement is reduced, increasing the software scalability.
\end{enumerate} 
We provide the mathematical framework including convergence analysis and error propagation.
\end{abstract}

\section{Introduction and related works}

Scientists have two broad sources of information: measurements and models. While measurements are equivalent to physical observations, the term models encompasses a set of parametric equations describing the space and time evolution of a number of physical variables. Models and observations are characterized by a key limitation: models involve approximations and simplifications, observations have spatio-temporal gaps, i.e. the observation acquisition space may be significantly different  from the model space  (indirect data), either in terms of  dimension or structure.  Data Assimilation (DA) adds value to the observations by filling in the gaps - by means of the so-called observation operator - and adds value to models by constraining them with observations - by using  (model-constrained) least square  methods. In this way, DA allows scientists to "make sense" of information: it provides mathematical methods for finding an optimal trade-off between the current estimate of the model’s state and the observations, at each time. In particular, we  will be concerned with DA mathematical methods tightly coupled with models - namely tis  UQ at the local levèsèèel by using Monte Carlo sampling (Q. Liao, K. Willcox, A domain decomposition approach for uncertainty analysis, (2010), A decomposition-based approach to uncertainty analysis of feed-forward multi component systems S. Amaral, D. me-dependent Partial Differential Equations (PDEs). In this work, such methods will be denoted tightly coupled PDE \& DA models. \\
\\
Main approaches for delivering scalable solutions of simulations based on DA methods integrated with a  PDE-based model essentially only takes full advantage of existing parallel PDE solvers, and in particular those  based on Domain Decomposition (DD) methods in space, where the DD-solver is suitably modified to also handle the adjoint system. Usually, iterative solvers are applied to solve the DA model. While this scheme is efficient, it has a limited scalability, due to the strong synchronization between the PDE integration and the DA solver. A different  approach is the combination of DD-methods in space and Uncertainty Quantification (UQ), where spatial domain-decomposed uncertainty quantification  approach performAllaire and K. Willcox, (2014), H. Antil,  M. Heinkenschloss,  R. H. W. Hoppe · D. C. Sorensen, Domain decomposition and model reduction for the numerical solution of PDE constrained optimization problems with localized optimization variables, (2010)). More recently, parallel  PDEàààà solvers based on DD in space-and-time were also proposed (M. Ulbriq, Generalized SQP-Methods with ”Parareal”, Time-Domain Decomposition for Time-dependent PDE-constrained Optimization(2004); J. Liua, Z. Wang, Efficient Time Domain Decomposition Algorithms for Parabolic PDE-Constrained Optimization Problems (2016)).  Finally, we mention the Parallel Data Assimilation Framework (PDAF, Nerger et al., 2005b, http://pdaf.awi.de) where parallel ensemble-based Kalman Filters algorithms are implemented and coupled within the PDE-model solver. However, parallelism is employed using a DD approach only across the spatial dimension (L. Nerger and W. Hiller, Software for ensemble-based data assimilation systems, Implementation strategies and scalability (2013)). Time-parallel approaches provide a new avenue to achieve scaling on new generation computing environments.
\\
A mathematical framework for next-generation extreme-scale computing is the space-and-time decomposition or PinT-based approach. European researchers are leading PinT developments, as evidenced by a series of international workshops dedicated to these algorithms held in Europe (Lugano, 2011, Manchester, 2013, and Jülich, 2014) with 21 European speakers. PinT methods are becoming increasingly popular for tackling the growing complexity of large scale high-fidelity simulations making better use of available computational resources for the solution of time-dependent PDEs. This is achieved by domain decomposition not only taking place along the spatial coordinates, but also on the time variable. Briefly, all of the PinT-based methods share this general idea: 
\begin{itemize}
\item use a coarse/global/predictor propagator to obtain approximate initial values of local models on the coarse time-grid; 
\item use a fine/local/corrector solver to obtain a more accurate solution of local models; 
\item apply an iterative procedure to smooth out the discontinuities of the global model. 
\end{itemize}
Nevertheless, one of the key limitation of scalability of any PinT-based methods is data dependencies of the coarse solver: the coarse solver must always be executed serially for the full duration of the simulation, the fine solver is applied in parallel to each interval after an initial condition is provided for it. Then, convergence is achieved when the value of the current correction falls below a certain prescribed tolerance.\\
On the contrary, the core of the proposed PinT-based approach is: 
\begin{enumerate} 
\item  DA acts as coarse/predictor for the local PDE -based forecasting model, increasing the accuracy of the local solution. 
\item The fine and coarse solvers are applied in parallel, increasing the efficiency of the algorithm.  \item Data locality is preserved and data movement is reduced, increasing the software scalability.
\end{enumerate} 

\section{The Parareal method applied to 4D-DA problem}
If $\Omega \subset \mathbb{R}^{3}$ is a spatial three dimensional domain, let:
\begin{equation}\label{modelloDA}
\left\{ \begin{array}{ll}
u(t_{2},x)=\mathcal{M}[u(t_{1},x)] & \textrm{$\forall x \in \Omega$, $t_{1},t_{2} \in [0,T]$, $(t_{2}>t_{1}>0)$} \\
u(t_{0},x)=u_{0}(x) & \textrm{$ t_{0}=0, \ \ x\in \Omega$}
\end{array}, \right.
\end{equation}
be a symbolic description of the predictive 4D-DA model of interest where
\begin{displaymath}
u:(t,x) \in [0,T] \times \Omega \mapsto u(t,x),
\end{displaymath} 
is the state function of $\mathcal{M}$, and let
\begin{displaymath}
v:(t,x) \in [0,T] \times \Omega \mapsto v(t,x),
\end{displaymath}
be the observations function, and
\begin{displaymath}
\mathcal{H}: u(t,x) \mapsto v(t,x), \ \ \ \ \ \forall (t,x) \in [0,T] \times \Omega,
\end{displaymath}
denote the non-linear observations mapping. \\
For the Variational DA (VarDA) formulation, we consider:
\begin{itemize}
\item $NP$ points of $\Omega \subset \mathbb{R}^{3}$ $:$ $\{x_{j}\}_{j=1,...,NP}\subset \Omega$;
\item $nobs$ points of $\Omega$, where $nobs <<NP$, $:$ $\{y_{j}\}_{j=1,...,nobs}$;
\item $N$ points of [0,T], $:$ $\{t_{k}\}_{k=1,...,N}$ with $t_{k}=t_{0}+k(h t)$;
\item the vector 
\begin{displaymath}
u_{0}=\{u_{0,j}\}_{j=1,...,NP}\equiv \{u(t_{0},x_{j})\}_{j=1,...,NP} \in \mathbb{R}^{NP},
\end{displaymath}
which is the state at time $t_{0}$;
\item the operator 
\begin{displaymath}
M_{k-1,k}\in \mathbb{R}^{NP \times NP}, \ \ \ k=1,...,N,
\end{displaymath}
representing a discretization of a linear approximation of $\mathcal{M}$ from $t_{k-1}$ to $t_{k}$ and for simplicity of notations, let us  
\begin{equation}\label{matriceM}
M\equiv M_{k-1,k};
\end{equation}
\item the vector 
\begin{displaymath}
\{u_{k,j}^{b}\}_{k=1,...,N-1;j=1,...,NP} \equiv \{u^{b}(t_{k},x_{j})\}_{k=1,...,N-1;j=1,...,NP} \in \mathbb{R}^{NP\times N-1},
\end{displaymath}
representing the solution of $M_{k-1,k}$ at $t_{k}$ for $k=1,...,N$, i.e. the background;
\item the vector
\begin{displaymath}
v_{k}\equiv \{v(t_{k},y_{j})\}_{j=1,...,nobs}\in \mathbb{R}^{N\times nobs},
\end{displaymath}
consisting of the observations at $t_{k}$, for $k=0,...,N-1$; 
\item the linear operator
\begin{displaymath}
H_{k}\in \mathbb{R}^{nobs \times NP}, \ \ \ k=0,...,N-1,
\end{displaymath}
representing a linear approximation of $\mathcal{H}$;
\item a block diagonal matrix $G\in \mathbb{R}^{(N \times nobs)\times (NP \times N)}$ such that
\begin{displaymath}
G=\left\{ \begin{array}{ll}
diag[H_{0},H_{1}M_{0,1},...,H_{N-1}M_{N-2,N-1} & \textrm{$N>1$} \\
H_{0}& \textrm{$ N=1$}
\end{array}, \right.
\end{displaymath}
\item \textbf{R} and \textbf{B}$=VV^{T}$ the covariance matrices of the errors on the observations and on the background,
respectively.
\end{itemize}
We now define the 4D-DA inverse problem \cite{articolo2}.
\begin{definizione}(The 4D-DA inverse problem). Given the vectors
\begin{displaymath}
v=(v_{k})_{k=0,...,N-1} \in \mathbb{R}^{N \times nobs}, \ \ u_{0} \in \mathbb{R}^{NP},
\end{displaymath}
and the block diagonal matrix
\begin{displaymath}
G \in \mathbb{R}^{(N \times nobs) \times (NP \times N)},
\end{displaymath}
a 4D-DA problem concerns the computation of
\begin{displaymath}
u^{DA}=(u_{k}^{DA})_{k=0,...,N-1} \in \mathbb{R}^{NP\times N},
\end{displaymath}
such that
\begin{equation}\label{DA}
v=G \cdot u^{DA},
\end{equation}
subject to the constraint that
$$u_{0}^{DA}=u_{0}.$$
\end{definizione}  
\noindent We also introduce the following definition of 4D-Var DA problem.
\begin{definizione}(The 4D-Var DA problem). The 4D-VarDA problem can be described as following:
\begin{equation}\label{varDA}
u^{DA}=argim_{u\in \mathbb{R}^{NP\times N}}J(u),
\end{equation}
with
\begin{equation}\label{funzionale}
J(u)=\alpha ||u-u_{0}||_{B^{-1}}^{2}+||Gu-v||_{R^{-1}}^{2},
\end{equation}
where $\alpha$ is regularization parameter.
\end{definizione}
We use in what follows a DD approach in \cite{tesi}, i.e. discrete MPS. \\ The discrete MPS is uses in \cite{tesi} for solving 3D-Var DA problem \cite{articolo}. 
\begin{definizione}(The 3D-Var DA problem). 3D Variational DA problem is to compute the vector $\textbf{u}^{DA}$ such that
\begin{equation}\label{3Dvar}
\textbf{u}^{DA}=argmin_{\textbf{u}\in \mathbb{R}^{NP}}\textbf{J}(\textbf{u})=argmin_{u}\left\{||\textbf{H}\textbf{u}-\textbf{v}||_{\textbf{R}}^{2}+\lambda ||\textbf{u}-\textbf{u}^{b}||_{\textbf{B}}^{2}\right\}
\end{equation}
where $\lambda$ is the regularization parameter.
\end{definizione}
The 3D-Var operator is:
\begin{equation}\label{3Doper}
\textbf{J}(\textbf{u})\equiv \textbf{J}(\textbf{u}, \textbf{R}, \textbf{B}, D_{NP}(\Omega))=(\textbf{H}\textbf{u}-\textbf{v})^{T}\textbf{R}(\textbf{H}\textbf{u}-\textbf{v})+ \lambda (\textbf{u}-\textbf{u}^{b})^{T}\textbf{B}(\textbf{u}-\textbf{u}^{b}).
\end{equation}
The matrix $\textbf{H}$ is ill conditioned so we consider the preconditioner matrix $\textbf{V}$ such that $\textbf{B}=\textbf{V}\textbf{V}^{T}$.
\\

So, the discrete MPS is composed of the following steps:
\begin{enumerate}
\item Decomposition of domain $\Omega$ into a sequence of sub domains $\Omega_{i}$ such that:
\begin{displaymath}
\Omega=\bigcup_{i=1}^{N}\Omega_{i}.
\end{displaymath}
\item Definition of interfaces of sub domains $\Omega_{i}$ as follows:
\begin{equation}\label{interfacce}
\Gamma_{ij}:=\partial \Omega_{i} \cap \Omega_{j} \ \ \ \ \textrm{for $i,j=1,...,J$}.
\end{equation}
\item Definition of restriction matrices $R_{i}$, $R_{ij}$ to sub domain $\Omega_{i}$ and interface $\Gamma_{ij}$, and extension matrices $R_{i}^{T}$, $R_{ij}$ to domain $\Omega$ for $i,j=1,...,J$ as follows:
\begin{equation}\label{RjDA}
R_{i}=\bordermatrix{  \footnotesize
& & &  & \textrm{ \footnotesize $s_{i-1}+1$} &\cdots & \textrm{ \footnotesize $s_{i-1}+r_{i}$} & & & \cr
& 0 & \cdots & 0& 0 & \cdots & 0 &0& \cdots  & 0  \cr
& \vdots & &\vdots & \vdots  &  & \vdots &\vdots & & \vdots  \cr
& 0 & \cdots & & 0 & \cdots & 0 &0 & \cdots &0 \cr
\textrm{ \footnotesize $s_{i-1}+1$}& 0 & \cdots & 0 & 1&  & & 0 &\cdots & 0\cr
\vdots & \vdots & & \vdots &  &\ddots & & &  & & \cr
\textrm{ \footnotesize $s_{i-1}+r_{i}$} & 0 &\cdots  &0 & 0 & &1 & 0 & \cdots & 0  \cr
& 0 & \cdots & 0& 0 & \cdots & 0 &0& \cdots  & 0  \cr
& \vdots & &\vdots & \vdots  &  & \vdots &\vdots & & \vdots  \cr
& 0 & \cdots & 0& 0 & \cdots & 0 &0 & \cdots &0 \cr
},
\end{equation}

\begin{equation}
R_{ij}=\bordermatrix{
& & &  & \textrm{ \footnotesize $\bar{s}_{i-1,i}+1$} &\cdots & \textrm{ \footnotesize $\bar{s}_{i-1,i}+r_{i}$} & & & \cr
& 0 & \cdots & 0& 0 & \cdots & 0 &0& \cdots  & 0  \cr
& \vdots & &\vdots & \vdots  &  & \vdots &\vdots & & \vdots  \cr
& 0 & \cdots & & 0 & \cdots & 0 &0 & \cdots &0 \cr
\textrm{ \footnotesize $\bar{s}_{i-1,i}+1$} & 0 & \cdots & 0 & 1&  & & 0 &\cdots & 0\cr
\vdots & \vdots & & \vdots &  &\ddots & & &  & & \cr
\textrm{ \footnotesize $\bar{s}_{i-1,i}+r_{i}$} & 0 &\cdots  &0 & 0 & &1 & 0 & \cdots & 0  \cr
& 0 & \cdots & 0& 0 & \cdots & 0 &0& \cdots  & 0  \cr
& \vdots & &\vdots & \vdots  &  & \vdots &\vdots & & \vdots  \cr
& 0 & \cdots & 0& 0 & \cdots & 0 &0 & \cdots &0 \cr
}
\end{equation}
where $s_{i,j}=r_{i}-C_{i,j}$, $\bar{s}_{i,j}=s_{i,j}+t_{ij}$, and $r_{i}$, $t_{i,j}$, $C_{i,j}$ points of sub domain $\Omega_{i}$, interfaces $\Gamma_{ij}$ and sub domain $\Omega_{ij}=\Omega_{i}\cap \Omega_{j}$, respectively.
\item  For $i=1,2,...,J$,  solution of $J$ subproblems $P_{i}^{n+1}$, for $n=0,1,2,...$ where 
\begin{equation}\label{argmin}
\begin{split}
P_{i}^{n+1} \ \ argmin_{\textbf{u}_{i}^{n+1}\in \mathbb{R}^{r_{i}}}\textbf{J}_{i}(\textbf{u}_{i}^{n+1}),
\end{split}
\end{equation}
where
\begin{equation}\label{funzionaliJi}
\textbf{J}_{i}(\textbf{u}_{i}^{n+1})=||\textbf{H}_{i}\textbf{u}_{i}^{n+1}-\textbf{v}_{i}||_{\textbf{R}_{i}}^{2}+\\ ||\textbf{u}_{i}^{n+1}-{({\textbf{u}_{i}}^{b})}||_{\textbf{B$_{i}$}}^{2}+ ||\textbf{u}_{i}^{n+1}/\Gamma_{ij}-\textbf{u}_{k}^{n}/\Gamma_{ij}||_{\textbf{B$/\Gamma_{ij}$}}^{2},
\end{equation}
as \textbf{B}$_{i}=R_{i}\textbf{B}R_{i}^{T}$ is a covariance matrix, we get that \textbf{B}/$\Gamma_{ij}=R_{i}\textbf{B}R_{ij}^{T}$ are the restriction of the matrix $B$, respectively, to the sub domain $\Omega_{i}$ and interface $\Gamma_{ij}$ in (\ref{interfacce}) $\textbf{H}_{i}=R_{i}\textbf{H}R_{i}^{T}$, ${\textbf{R}_{i}}=R_{i}\textbf{R}R_{i}^{T}$ the restriction of the matrices $\textbf{H}$, ${\textbf{R}}$ to the sub domain $\Omega_{i}$, $\textbf{u}_{i}^{b}=R_{i}\textbf{u}^{b}$, $\textbf{u}_{i}^{n+1}/\Gamma_{ij}=R_{ij}\textbf{u}_{i}^{n+1}$, $\textbf{u}_{j}^{n}/\Gamma_{ij}=R_{ij}\textbf{u}_{j}^{n}$ the restriction of vectors $\textbf{u}^{b}$, $\textbf{u}_{i}^{n+1}$, $\textbf{u}_{j}^{n}$ to the sub domain $\Omega_{i}$ and interface $\Gamma_{ij}$,  for $i,j=1,2,...,J$. \\
\\
The MPS in \cite{MPS} is used for solving boundary-value problems and as transmission condition on interfaces $\Gamma_{ij}$ for $i,j=1,...,J$ it requires that solution of subproblem on $\Omega_{i}$ at iteration $n+1$ coincides with solution of subproblem on adjacent sub domain $\Omega_{j}$ at iteration $n$; but the 3D-Var DA problem is a variational problem. So, according MPS, we impose the minimization in norm $||\cdot ||_{\textbf{B}/{\Gamma_{ij}}}$ between $\textbf{u}_{i}^{n+1}$ and $\textbf{u}_{j}^{n}$.
The functional $\textbf{J}$ defined in (\ref{3Doper}) as well as all the functionals $\textbf{J}_{i}$ defined in (\ref{funzionaliJi}), are quadratic (hence, convex), so
their unique minimum are obtained as zero of their gradients. In particolar, the functional  $\textbf{J}_{i}$ can be rewritten as follows:
\begin{displaymath}
\begin{split}
\frac{1}{2}(\textbf{w}_{i}^{n+1})^{T}\textbf{w}_{i}^{n+1}+ \frac{1}{2}(\textbf{H}_{i}\textbf{V}_{i}\textbf{w}_{i}^{n+1}-\textbf{d}_{i})^{T}\textbf{R}_{i}^{-1}(\textbf{H}_{i}\textbf{V}_{i}\textbf{w}_{i}^{n+1}-\textbf{d}_{i})+\\ \frac{1}{2}(\textbf{V}_{ij}\textbf{w}_{i}^{n+1}-\textbf{V}_{ij}\textbf{w}_{j}^{n+1})^{T} \cdot (\textbf{V}_{ij}\textbf{w}_{i}^{n+1}-\textbf{V}_{ij}\textbf{w}_{j}^{n}),
\end{split}
\end{displaymath}
where $\textbf{w}_{i}^{n+1}=\textbf{V}_{i}^{T}({\textbf{u}_{i}^{n+1}}-{\textbf{u}_{i}^{b}})$, $\textbf{V}_{i}=R_{i}\textbf{V}R_{i}^{T}$ is the restriction of matrix $\textbf{V}$ to sub domain $\Omega_{i}$, $\textbf{V}_{ij}=R_{i}\textbf{V}R_{ij}^{T}$ is the restriction of matrix $\textbf{V}$ to interfaces $\Gamma_{ij}$, $\textbf{d}_{i}$ the restriction of vector $\textbf{d}=[\textbf{v}-\textbf{H}(\textbf{u})]$. The gradients of $\textbf{J}_{i}$ is:
\begin{equation}
\nabla \textbf{J}_{i}(\textbf{w}_{i}^{n+1})=\textbf{w}_{i}^{n+1}+\textbf{V}_{i}^{T}\textbf{H}_{i}^{T}\textbf{R}_{i}^{-1}(\textbf{H}_{i}\textbf{V}_{i}\textbf{w}_{i}^{n+1}-\textbf{d}_{i})+\textbf{V}_{ij}^{T}(\textbf{V}_{ij}\textbf{w}_{i}^{n+1}-\textbf{V}_{ij}\textbf{w}_{j}^{n})
\end{equation}
that can be rewritten as follows
\begin{equation}\label{gradiente}
\nabla \textbf{J}_{i}(\textbf{w}_{i}^{n+1})=(\textbf{V}_{i}^{T}\textbf{H}_{i}^{T}\textbf{R}_{i}^{-1}\textbf{H}_{i}\textbf{V}_{i}+I_{i}+B/\Gamma_{ij})\textbf{w}_{i}^{n+1}-\textbf{c}_{i}+B/\Gamma_{ij} \textbf{w}_{j}^{n},
\end{equation}
where 
\begin{equation}\label{ciMPS}
{c}_{i}=(\textbf{V}_{i}^{T}\textbf{H}_{i}^{T}\textbf{R}_{i}^{-1}\textbf{H}_{i}\textbf{V}_{i}\textbf{d}_{i}),
\end{equation} 
and $I_{i} \in \mathbb{R}^{r_{i} \times r_{i}}$  the identity matrix. \\ From (\ref{gradiente}) by considering the Euler-Lagrange equations we obtain the following systems ${(S_{i}^{MPS})}^{n+1}$:
\begin{equation}\label{DAP}
{(S_{i}^{MPS})}^{n+1}:  \ \ \ A_{i}^{MPS}\textbf{w}_{i}^{n+1}={c}_{i}-\sum_{j \neq i}A_{i,j}\textbf{w}_{j}^{n}, 
\end{equation}
to solve for $n=0,1,...$, where
\begin{equation}\label{AiMPS}
A_{i}^{MPS}=(\textbf{V}_{i}^{T}\textbf{H}_{i}^{T}\textbf{R}_{i}^{-1}\textbf{H}_{i}\textbf{V}_{i}+I_{i}+\textbf{B}/\Gamma_{ij}),
\end{equation}
and $A_{ij} =\textbf{B}/\Gamma_{ij}$, for $i,j=1,...,J$.
\item  For $i=1,...,J$, computation of  $\textbf{u}_{i}^{n+1}$, related to the sub domain $\Omega_{i}$, as follows:
\begin{equation}\label{solDAi}
u_{i}^{MPS,n+1}\equiv \textbf{u}_{i}^{n+1}=\textbf{u}_{i}^{b}+\textbf{B}_{i}^{-1} \textbf{V}_{i}\textbf{w}_{i}^{n+1}\\.
\end{equation}
\item Computation of $\textbf{u}^{DA}$, solution of 3D-Var DA problem in (\ref{3Dvar}), obtained by patching together all the vectors $\textbf{u}_{i}^{DA}$, i.e.:
\begin{equation}\label{uda}
u^{MPS}(x_{j}) \equiv \textbf{u}^{DA}(x_{j})=\left\{\begin{array}{ll} \textbf{u}_{i}^{m}(x_{j}) & \textrm{se $x_{j} \in \Omega_{i}$}\\
\textbf{u}_{k}^{m}(x_{j}) & \textrm{se $x_{j} \in \Omega_{k}$ o $x_{j} \in \Omega_{i} \cap \Omega_{k} $}, \end{array}, \right.
\end{equation}
for $i,k=1,...J$, and $m$ corresponding iterations needed to stop of the iterative procedure.
\end{enumerate}
The Parareal method was presented by J. L. Lions, Y. Maday, and
G. Turinici in \cite{Lions} as a numerical method to solve evolution problems in parallel. The
name was chosen to indicate that the algorithm is well suited for parallel real time
computations of evolution problems whose solution cannot be obtained in real time
using one processor only. In particolar, the Parareal method is a technique for solving general partial differential equations \cite{Martin},
this method has received some attention and a presentation under the format of
a predictor-corrector algorithm has been made by  G. Bal, Y. Maday in \cite{Bal} and
also by L. Baffico et al. in \cite{Baf}. It is this last presentation that we shall use in what
follows.
\\
The Parareal scheme uses the decomposition of time interval [0,T] to define the subproblems, and it defines the boundary
conditions compatible with the initial condition for each local problems.
\\
It's scheme is composed by two steps:
\begin{itemize}
\item First step: decomposition of interval of time [0,T]
\begin{displaymath}
[0,T]=\bigcup_{k=1}^{N-1}[t_{k-1},t_{k}],
\end{displaymath}
where $t_k$ are $N$ points of [0,T] and $t_{0}=0$, $t_{N}=T$. Computation of $u_{k}^{b,n+1}$ for $k=1,...,N-1$, such that
\begin{equation}\label{co}
u_{k}^{b,n+1}=M\cdot u_{k-1}^{n+1},\ \ \ \textrm{for $n=0,1,...$}
\end{equation}
where $M$ is the matrix in (\ref{matriceM}) and $\{u_{k}^{b,n+1}\}_{k=0,...,N-1}$ is the background.
\item Second step: decomposition of domain $\Omega$
\begin{displaymath}
\Omega=\bigcup_{i=1}^{N_{sub}}\Omega_{i},
\end{displaymath}
and $N_{sub}$ sub domains $\Omega_{i}\subset \Omega$.\\ Let $x\in \mathbb{R}^{NP}$ be a vector, for simplicity of notations, we refer to $x_{i}$ as a restriction of $x$ to $\Omega_{i}$, i.e. $x_{i}\equiv x/\Omega_{i}$, similarly for matrix $A\in \mathbb{R}^{NP \times NP}$, i.e. $A_{i}\equiv A/\Omega_{i}$, according the description in \cite{articolo2}. $\forall i=1,...,N_{sub}$ and $k=1,...,N-1$ let:
\begin{equation}\label{problems}
P_{i,k}^{n} \ \ \ \left\{ \begin{array}{ll}
u_{i}^{DA,n}=argminJ_{i}(u_{i}) \\
u_{i,k-1}^{DA,n}=u_{i,k-1}^{b,n}
\end{array}, \ \ \ \ \ \ \ \textrm{for } n=1,2,... \right.
\end{equation}
with
\begin{displaymath}
J_{i}(u_{i})=J(u)/\Omega_{i}+\rho\sum_{k=0}^{N-1}||(M_{i})^{k}u_{i}/\Omega_{ij}-(M_{i})^{k}u_{j}/\Omega_{ij}||_{B_{ij}^{-1}}^{2},
\end{displaymath}
and
\begin{displaymath}
J(u)/\Omega_{i}=||u_{i}-u_{i,0}||_{B_{i}^{-1}}+||G_{i}u_{i}-v_{i}||_{R_{i}^{-1}}^{2}.
\end{displaymath}
be a the local DA problem. \\ 
By setting $w_{i}=V_{i}^{T}(u_{i}^{MPS}-u_{i}^{b})$, we can apply the MPS to $P_{i,k}^{n}$, i.e. we solve $\forall i=1,...,N_{sub}$ and $k=1,...,N-1$ the following systems:
\begin{equation}\label{MPS}
A_{i}^{MPS}w_{i}^{n}=c_{i}-(\textbf{B}_{ij}M/\Omega_{j})w_{j}^{n}, \ \ \ n=1,2,...,
\end{equation}
where
\begin{align*}
A_{i}^{MPS}&=(\textbf{V}_{i}^{T}\textbf{G}_{i}^{T}\textbf{R}_{i}^{-1}\textbf{G}_{i}\textbf{V}_{i}+I_{i}+\textbf{B}/\Gamma_{ij}M/\Omega_{j}),\\
c_{i}&=(\textbf{V}_{i}^{T}\textbf{G}_{i}^{T}\textbf{R}_{i}^{-1}\textbf{G}_{i}\textbf{V}_{i}\textbf{d}_{i}),
\end{align*}
and $\textbf{d}_{i}=(\textbf{v}_{i}-\textbf{G}_{i}\textbf{u}_{i}^{MPS})$.
\end{itemize}
According to Parareal $\forall i=1,...,N_{sub}$ and $k=1,...,N-1$, numerical solution of $P_{i,k}^{n+1}$ in (\ref{problems}) is:
\begin{equation}\label{solpara}
\begin{array}{ll}
u_{i,k+1}^{n+1}&=u_{i,k+1}^{b,n+1}+\textbf{B}_{i}^{-1}V_{i}w_{i}^{n}\\
&= M\cdot u_{i,k}^{n+1}+u_{i,k+1}^{MPS,n}-u_{i,k+1}^{b,n}
\end{array}\textrm{for $n=1,2,...$}.
\end{equation}
By suitably reorganizing the $NP$ points of $\Omega$, for $k=1,...,N-1$ the numerical solution $u_{k+1}^{n+1}$ of the 4D-DA inverse problem defined in (\ref{varDA}) is:
\begin{equation}
\begin{array}{ll}
u_{k+1}^{n+1}&=[u_{1,k+1}, u_{2,k+1},...,u_{N_{sub},k+1}]'\\
&=  M\cdot u_{k}^{n+1}+u_{k+1}^{MPS,n}-u_{i,k+1}^{b,n},
\end{array} \textrm{for $n=1,2,...$}
\end{equation}
and from the (\ref{co}), $u_{k+1}^{n+1}$ can be rewritten as follows
\begin{equation}\label{solparatot}
u_{k+1}^{n+1}=M\cdot u_{k}^{n+1}+MPS(u_{k}^{n})-M\cdot u_{k+1}^{n},
\end{equation} 
where $u_{k+1}^{MPS,n}\equiv MPS(u_{k}^{n})$. 
\begin{lemma}\label{lemma}
Let $N\in \mathbb{N}$ and $R>0$, $H\ge0$. If for $k=0,1,...,N$ we have that: $$|M_k|\le (1+R)|M_{k-1}|+H \;\;\;\text{for}\; k=1,2,...,N$$
	then it holds that
	$$|M_k|\le e^{NR}|M_0|+\dfrac{e^{NR}-1}{R}H \;\;\;\text{for}\; k=1,2,...,N.$$
\end{lemma}
\noindent In the following we assume  $||\cdot ||\equiv ||\cdot ||_{\infty}$.
\begin{lemma}\label{lemma1}
Let $M$ a discretization of a linear approximation of $\mathcal{M}$ in (\ref{modelloDA}), $\mu(M)$ its condition number, $N\in \mathbb{N}$, $\forall k=1,...,N$ and $u,v \in \mathbb{R}^{N}$ then it is 
 $$||M\cdot u_{k-1}-M\cdot v_{k-1}|| \le C \frac{1}{\mu (A)}$$ is $A$ the Hessian of the operator $J$ defined  in (\ref{funzionale}) and $C$ constant against $\mu (A)$.
\end{lemma}
\begin{proof}
For $k=1,...,N$ let $t_k$ be fixed. We have that:
\\
\begin{equation*}
\begin{split}
||M\cdot u_{k-1}-M\cdot v_{k-1}|| &\le ||M|| ||u_{k-1}-v_{k-1}||\\ &=||M||\cdot ||M^{T}||^{-1}||M^{T}||\cdot ||u_{k-1}-v_{k-1}||,
\end{split}
\end{equation*}
as in \cite{articolo2} we let:
\begin{equation*}
\mu(M) \ge ||M^{T}||^{-1}
\end{equation*}
then we get 
\begin{equation}\label{rel2}
\begin{split}
||M\cdot u_{k-1}-M\cdot v_{k-1}|| &\le ||M||\cdot ||M^{T}||^{-1}||M^{T}||\cdot ||u_{k-1}-v_{k-1}||\\ & \le \mu(M)||M||\cdot ||M^{T}||\cdot ||u_{k-1}-v_{k-1}||.
\end{split}
\end{equation}
Let $\sigma$, $\xi$, $\delta$ be the errors on  $u_{N}^{b}$, $u_{0}^{b}$, $u^{DA}$. According to the assumptions used in \cite{articolo2}:
\begin{equation}\label{rela}
\begin{array}{ll}
||\sigma||&=\mu(M)||\xi|| \longrightarrow \mu(M)=\frac{||\sigma||}{||\xi||};\\
||\delta||&=\mu(J)||\sigma|| \longrightarrow \mu(J)=\frac{||\delta||;}{||\sigma||}\\
\mu(J)&=\mu(A);
\end{array}
\end{equation}
it is
\begin{align}\label{rel1}
\mu(M)=\frac{||\delta||}{||\xi||} \cdot \frac{1}{\mu(A)}.
\end{align}
We can note that $$||M||=||M^{T}||,$$ and if we let $L=||M||^{2}$ from (\ref{rel2}) it comes out that:
\begin{equation}\label{rel3}
||M\cdot u_{k-1}-M\cdot v_{k-1}|| \le L \mu(M) ||u_{k-1}-v_{k-1}||;
\end{equation}
by replacing the (\ref{rel1}) in (\ref{rel3}) the thesis follows where $C=L\cdot \frac{||\delta||}{||\xi||}$.
\end{proof}
\\
\\
\noindent Finally we are able to prove the following result.
\begin{proposizione}\label{convergenza}
Let $u^{DA}$ be the solution of the 4D-Var DA problem in (\ref{varDA}) and $\forall k=0,...,N$ $u_{k}^{n}$ in (\ref{solparatot}) the solution obtained by applying the Parareal method with $n$ iterations and let:
\begin{equation}\label{errore}
\begin{array}{ll}
\delta(u_{k}^{n})&=u_{k}^{MPS,n}-u_{k}^{b,n}\\
&=MPS(u_{k-1}^{n})-M\cdot u_{k-1}^{n}
\end{array}  \ \ \ \textrm{for $k=1,...,N$},
\end{equation}
be the correction factor on $t_{k}$ with $n$ iterations.\\ Let us assume that:
\begin{enumerate}
\item $\forall k=1,...,N$ \begin{equation}\label{punto1} u_{k}^{DA}\equiv u^{DA}(t_{k})=MPS(u_{k-1}^{DA});
\end{equation}
\item $\forall k=1,...,N$ and $u,v\in \mathbb{R}^{N}$
\begin{equation}\label{punto2}
||M\cdot u_{k-1}-M\cdot v_{k-1} || \le C \frac{1}{\mu (A)},
\end{equation}
where $M$ is given as in (\ref{matriceM}), $A$ is the Hessian of the operator $J$ in (\ref{funzionale}) and $C$ constant against $\mu (A)$;
\item Let $E_{k}^{b}(h)$ be initial error of DA on $t_{k}$, $\forall k=1,...,N-1$ we have that:
\begin{displaymath}
|E_{k}^{b}(h)|=||u_{k}^{DA}-M\cdot u_{k-1}|| \le C(h),
\end{displaymath}
where $C(h)= \mathcal{O}(h^{p})$, $p$ is order of convergence and $h$ is step-size of [0,T], i.e. the numerical scheme  applied for discretizing of the model $\mathcal{M}$ in (\ref{modelloDA}) is convergence for $h$.
\end{enumerate}
Then $\forall k=1,...,N$, it holds that
\begin{equation}\label{tesi}
|E_{k}^{n}(h)|=||u_{k}^{DA}-u_{k}^{n}|| \le c_{n}(h), \ \ \ \textrm{for $n=1,2,...$}
\end{equation}
where $c_{n}(h)=\mathcal{O}(h^{p})$.
\end{proposizione}
\begin{proof}
The DD method used in first step of Parareal method, namely the MPS, satisfies the (\ref{punto1}) as it is proved in \cite{tesi}, while the (\ref{punto2}) is proved in Lemma \ref{lemma1}.
\\
We prove the thesis using induction on $n$.
\\
 \underline{Base case.} $n=1$ numerical solution is given by using Parareal, then  $\forall k=1,...,N$ it is:
$$|E_k^1(h)|=\|u_k^{DA}-u_k^{b,1}\|=\|u_k^{DA}-M\cdot u_{k-1}\| {\le} C(h).$$
\\
\underline{Induction step.} It holds that
\begin{equation}\label{passo}
|E_k^n(h)|=\|u_k^{DA}-u_k^{n}\|\le c_{n}(h) \ \ \ \forall k=1,...,N,
\end{equation}
we prove it for $n+1$, i.e. we will prove that
\begin{equation}\label{inductionstep}
 |E_k^{n+1}(h)|=\|u_k^{DA}-u_k^{n+1}\|\le c_{n+1}(h), \ \ \forall k=1,...,N.
 \end{equation}
\\We rewrite $u_k^{DA}$ by using (\ref{punto1}) and $u_k^{n+1}$ by using (\ref{solparatot}) and (\ref{errore}), by using:
\begin{itemize}
	\item $u_k^{DA}{=}MPS(u_{k-1}^{DA})$;
	\item $ u_k^{n+1}=M\cdot u_{k-1}^{n+1}+MPS(u_{k-1}^{n})-M\cdot u_{k}^{n}=M\cdot u_{k-1}^{n+1}+\delta(u_{k-1}^n)$ for $n=1,2,...$
\end{itemize}
so, we have that
\begin{align*}
u_k^{DA}-u_k^{n+1}&=MPS(u_{k-1}^{DA})-M\cdot u_{k-1}^{n+1}-\delta(u_{k}^n)\\
&=MPS(u_{k-1}^{DA})-M\cdot u_{k-1}^{DA}+M\cdot u_{k-1}^{DA}-M\cdot u_{k-1}^{n+1}-\delta(u_{k}^n)\\
&=\delta(u_{k}^{DA})-\delta(u_{k}^n)+M\cdot u_{k-1}^{DA}-M\cdot u_{k-1}^{n+1}
\end{align*}
and from (\ref{punto2}) and base case,
\begin{equation}\label{errorefi}
\begin{array}{ll}
|E_k^{n+1}(h)|&=\|u_k^{DA}-u_k^{n+1}\|\\
&\le \|\delta(u_{k}^{DA}) -\delta(u_{k}^n)\|+\|M\cdot u_{k-1}^{DA}-M\cdot u_{k-1}^{n+1}\|\\
&=||(u_{k}^{DA}-M\cdot u_{k-1}^{DA})-(MPS(u_{k-1}^{n})-M\cdot u_{k-1}^{n})||+ \|M\cdot u_{k-1}^{DA}-M\cdot u_{k-1}^{n+1}\|\\
& \le ||u_{k}^{DA}-MPS(u_{k-1}^{n})||+||M\cdot u_{k-1}^{DA} -M\cdot u_{k-1}^{n}|| +\|M\cdot u_{k-1}^{DA}-M\cdot u_{k-1}^{n+1}\|.
\end{array}
\end{equation}
In \cite{tesi} convergence of MPS is demonstrated, i.e.
\begin{equation}\label{conv}
\forall \epsilon^{MPS} \ \ \exists M(\epsilon^{MPS})>0 : |n|>M(\epsilon^{MPS}) \Rightarrow ||u_{k}^{DA}-MPS(u_{k-1}^{n})||<\epsilon^{MPS}, 
\end{equation} 
so from (\ref{conv}) and (\ref{punto2}) the (\ref{errorefi}) can be rewritten as follows
\begin{align*}
|E_k^{n+1}(h)|&=\|u_k^{DA}-u_k^{n+1}\|\\
&{\le} \epsilon^{MPS}+||M\cdot u_{k-1}^{DA}-M\cdot u_{k-1}^{n}||+||M\cdot u_{k-1}^{DA}-M\cdot u_{k-1}^{n+1}||\\
& {\le} C\frac{1}{\mu(A)} (\|u_{k-1}^{DA}-u_{k-1}^n\|+\|u_{k-1}^{DA}-u_{k-1}^{n+1}\|)+\epsilon^{MPS}\\
&= C\frac{1}{\mu(A)} ( |E_{k-1}^n(h)|+|E_{k-1}^{n+1}(h)|)+\epsilon^{MPS}\\
\end{align*}
and from (\ref{passo}) it follows that
\begin{align*}
&{\le}  C\frac{1}{\mu(A)} c_{n}(h)+C\frac{1}{\mu(A)}|E_{k-1}^{n+1}(h)|+\epsilon^{MPS}\\
&=C\frac{1}{\mu(A)}|E_{k-1}^{n+1}(h)|+ C\frac{1}{\mu(A)} \cdot c_{n}(h)+\epsilon^{MPS}.
\end{align*}
We apply Lemma \ref{lemma}, where $R\equiv R_{\mu(A)}=\frac{C-\mu(A)}{\mu(A)}$ and $H= C\frac{1}{\mu(A)} c_{n}(h)+\epsilon^{MPS}$, so we have that:
\begin{displaymath}
|E_k^{n+1}(h)|\le e^{NR_{\mu(A)}}|E_0^{n+1}(h)|+\dfrac{e^{NR_{\mu(A)}}-1}{R_{\mu(A)}}H.
\end{displaymath}
Finally, supposed that the error at time $t_{0}=0$ is null, i.e. $E_0^{n+1}(h)=0$, we have that: 
\begin{equation}\label{rel}
|E_k^{n+1}(h)|\le \dfrac{e^{NR_{\mu(A)}}-1}{R_{\mu(A)}} (C\frac{1}{\mu(A)} \cdot c_{n}(h)+\epsilon^{MPS})
\end{equation}
assuming that
\begin{equation}\label{c} 
c_{n+1}(h)= \dfrac{e^{NR_{\mu(A)}}-1}{R_{\mu(A)}}\big( C\frac{1}{\mu(A)} \cdot c_{n}(h)+\epsilon^{MPS}\big)
\end{equation}
and the (\ref{inductionstep}) follows, i.e. the thesis (\ref{tesi}) for $n+1$ iterations.
\end{proof}
\\
\\

\noindent \underline{Remark:} Consider the behavior of $c_{n+1}$ in (\ref{c}) when $\mu(A)$ increases.  It holds that: 
 $$R_{\mu(A)}=\frac{C-\mu(A)}{\mu(A)} \thickapprox -1\ \ \  \textrm{(as $\mu(A)$ increases)}$$ and
\begin{equation}\label{termine} \dfrac{e^{NR_{\mu(A)}}-1}{R_{\mu(A)}} \thickapprox 1-\frac{1}{e^{N}} \ \ \ \textrm{(as $\mu(A)$ increases)}.
\end{equation}
We note that $$1-\frac{1}{e^{N}} \thickapprox 1 \,\,\, \textrm{ (as $N$ grows  i.e. as $h$ (which is the step-size of interval [0,T]) decreases.)}$$ Finally,  as $\frac{1}{\mu(A)} \thickapprox 0$ it follows that  $$c_{n+1}(h) \thickapprox \epsilon^{MPS}$$ then from (\ref{rel}) we get the convergence of  Parareal.
\\
\\
\noindent Now, we consider the local and global roundoff errors.
\begin{definizione}
Let $u_k^{n+1}=M\cdot u_{k-1}^{n+1}+\delta(u_{k-1}^n)$, with $k=0,...,N$ be the numerical solution obtained by Parareal method at $(n+1)$ iterations, and $\tilde{u}_n^{n+1}=M\cdot \tilde{u}_{k-1}^{n+1}+\delta(\tilde{u}_{k-1}^n)+\rho_k$  is the corresponding floating point representation, where $\rho_k$ is \textit{local round-off error}. 
	\\Fixed $n$, let $$|R_k^{n+1}(\mu{(A)})|=\|u_k^{n+1}-\tilde{u}_k^{n+1}\|$$
	for \textrm{$k=1,...,N$} be \textit{global round-off error} on $t_{k}$.
\end{definizione}
Under the assumptions of Proposition \ref{convergenza} and fixed the iteration $n+1$, we have
\begin{align*}
|R_k^{n+1}(\mu{(A)})|&=\|u_k^{n+1}-\tilde{u}_k^{n+1}\|=\|M\cdot u_{k-1}^{n+1}+\delta(u_{k}^n)-M\cdot\tilde{u}_{k-1}^{n+1}-\delta(\tilde{u}_{k}^n)-\rho_k\|\\&\le \|M\cdot u_{k-1}^{n+1}-M\cdot \tilde{u}_{k-1}^{n+1}\|+\|\delta(u_{k}^n)-\delta(\tilde{u}_{k}^n)\|+|\rho_k|\\
&\le C\frac{1}{\mu(A)} ||u_{k-1}^{n+1}-\tilde{u}_{k-1}^{n+1}\|+\|u_{k-1}^n-\tilde{u}_{k-1}\|+2|\rho_k|\\
&\le (C\frac{1}{\mu(A)} +1)||u_{k-1}^{n+1}-\tilde{u}_{k-1}^{n+1}\|+ 2\rho\\
& \le  (C\frac{1}{\mu(A)}+1)|R_{k-1}^{n+1}(\mu{(A)})|+2 \rho
\end{align*}
where
\begin{equation}\label{rho}
\rho=\max_{k=1,...,N}|\rho_k|.
\end{equation}
\\From Lemma \ref{lemma}, with $R\equiv R_{\mu(A)}=\frac{C-\mu(A)}{\mu(A)}$ and $H= (C\frac{1}{\mu(A)}+1)|R_{k-1}^n(\mu{(A)})|+2\rho$ it follows that:
\begin{equation*}
|R_k^{n+1}( \mu(A))|\le e^{NR_{\mu(A)}}|R_0^{n+1}(\mu{(A)})|+\dfrac{e^{NR_{\mu(A)}}-1}{R_{\mu(A)}}[(C\frac{1}{\mu(A)}+1)|R_{k-2}^n(\mu{(A)})|+2\rho].
\end{equation*}
then it is
\begin{equation}\label{roundoff}
\begin{split}
|R_k^{n+1}(\mu{(A)})|\le e^{NR_{\mu(A)}}|R_0^{n+1}(\mu(A))|+\dfrac{e^{NR_{\mu(A)}}-1}{R_{\mu(A)}}[(C\frac{1}{\mu(A)}+1)|R_{k-1}^n(\mu{(A)})|]+\\ \dfrac{e^{NR_{\mu(A)}}-1}{R_{\mu(A)}} 2\rho.
\end{split}
\end{equation}
Relation (\ref{roundoff}) is made of three terms: the first rapresents the propagation error on the initial value, the second rapresents the propagation of the round-off error during the iterations, the last term rapresents the dependence of $R_{\mu(A)}=\frac{C-\mu(A)}{\mu(A)}$. \\
The matrix $A$ is ill conditioned so as $\mu(A)$ increases it follows that $$\dfrac{e^{NR_{\mu(A)}}-1}{R_{\mu(A)}}\approx 1-\frac{1}{e^{N}}$$ and if $N$ increases, i.e. if $h$ decreases, then $$1-\frac{1}{e^{N}} \rightarrow  1.$$
\\
So, when the Parareal method is used, the roundoff error are not amplified as the number of iteration grows, because by using a  suitable value of $h$, the roundoff error at iteration $n+1$ is smaller than the sum of roundoff error at iteration $n$ and $2\rho$, where $\rho$ defined in (\ref{rho}). 
\section{Conclusion}
 DD Pint-based  methods allow the reformulation of VArDA  problem on a partition of the computational domain into subdomains. As such, it provides a very convenient framework for the solution of heterogeneous or multiphysics problems, i.e. those that are governed by differential equations of different kinds in different subregions of the computational domain. The effectiveness of PinT-based approaches are often dependent on the coarse grid operator (predictor) and intergrid operator (corrector). To this regard, our  approach uses the strong relationship of tightly coupled PDE\&DA: it   leads to a layered or hierarchical decomposition, which can be beneficial when matched with the expected hierarchical nature of upcoming exascale computing architectures.   Furthermore, the hierarchical decomposition may be applied globally across the simulation domain or locally, as in adaptive mesh and algorithm refinement, to restrict consideration of the finest scale to only those regions where such a description is important. The benefits of a layered algorithmic arrangement for exascale computing originate in the expected layered architectural arrangement of upcoming exascale computers.
\\
\\

\end{document}